\documentclass[11pt,reqno,letterpaper]{amsart}

\usepackage{amssymb}
\usepackage{amsmath}
\usepackage{amsthm}
\usepackage{bbm}
\usepackage{mathrsfs}
\usepackage{doi}
\usepackage{graphicx}
\usepackage[shortlabels]{enumitem}
\usepackage[amsstyle]{cases}
\usepackage{xcolor}
\usepackage{comment}
\usepackage{tikz-cd}

\usepackage{bookmark}
\usepackage{hyperref}
\usepackage{mathtools}
\usepackage{cleveref}
\hypersetup{pdfstartview={FitH}}

\newcommand{\SF}{\mathcal{SF}} % placeholder choice of notation, feel free to modify
\newcommand{\N}{\mathbb{N}}
\newcommand{\R}{\mathbb{R}}
\newcommand{\Z}{\mathbb{Z}}
\newcommand{\eps}{\varepsilon}

\numberwithin{equation}{section}

\newtheorem{theorem}{Theorem}
\newtheorem{lemma}[theorem]{Lemma}
\newtheorem{proposition}[theorem]{Proposition}

\theoremstyle{definition}
\newtheorem{definition}{Definition}
\newtheorem{remark}[theorem]{Remark}

\begin{document}

\title{Growth rates of sequences governed by the squarefree properties of their translates}

\author[Wouter van Doorn]{Wouter van Doorn}
\address{Wouter van Doorn, Groningen, the Netherlands} 
\email{wonterman1@hotmail.com}

\author[Terence Tao]{Terence Tao}
\address{Terence Tao, Department of Mathematics, UCLA, 405 Hilgard Ave, Los Angeles, USA}
\email{tao@math.ucla.edu}

 %NT - Distribution of integers with specified multiplicative constraints

%\date{\today}


\begin{abstract} We answer several questions of Erd\H{o}s regarding sequences of natural numbers $A$ whose translates $n+A$ intersect with the squarefree numbers in various specified ways. For instance, we show that if every translate only contains finitely many squarefree numbers, then $A$ has zero density, although the decay rate of this density can be arbitrarily slow. On the other hand, there exist sequences $A$ with optimal density $6/\pi^2$ for which infinitely many $n$ exist such that $n+a$ is squarefree for all $a \in A$ with $a < n$. In fact, infinitely many such $n$ exist for every exponentially increasing sequence, as long as the sequence avoids at least one residue class modulo $p^2$ for all primes $p$, a property we call admissible. If one instead requires infinitely many $n$ to exist such that $n+a$ is squarefree for all $a \in A$, then $A$ can have density arbitrarily close to, but not equal to, $6/\pi^2$. Finally, we prove bounds on the growth rate of sequences $A$ for which $a+a'$ is squarefree for all $a,a' \in A$, as well as bounds on the largest admissible subset of $\{1, 2, \ldots, N\}$. 
\end{abstract}

\maketitle

%%%%%%%%%%%%%%%%%%%%%%%%%%%%%%%%%%%%%%%%%%%%%%%%

\section{Introduction} \label{intro}
In $1981$, Erd\H{o}s \cite{ER81} stated various unconventional questions on additive and multiplicative number theory. In particular, he asked about the growth rates of sequences of positive integers whose translates are constrained by various properties concerning the squarefree natural numbers
$$ \SF \coloneqq \N \backslash \bigcup_p p^2 \N = \{1,2,3,5,6,7,10,\dots\} \quad (\href{https://oeis.org/A005117}{\text{OEIS A005117}}).$$
As is well known, by Euler's solution to the Basel problem this set has natural density
\begin{equation}\label{basel}
 \prod_p \left(1 - \frac{1}{p^2} \right) = \frac{1}{\zeta(2)} = \frac{6}{\pi^2} \quad (\href{https://oeis.org/A059956}{\text{OEIS A059956}}),
\end{equation}
and it is not hard to show\footnote{See Section \ref{notation-sec} for our notation conventions.}
$$|\SF \cap [x]| = \frac{6}{\pi^2} x + O(\sqrt{x}) \quad (\href{https://oeis.org/A013928}{\text{OEIS A013928}}).$$
In fact, this error term can be improved, and one has
\begin{equation}\label{sfx}
|\SF \cap [x]| = \frac{6}{\pi^2} x + O\left(\frac{\sqrt{x}}{\exp(c \log^{3/5} x / (\log\log x)^{1/5} )}\right)
\end{equation}
for all $x \geq 3$ and some absolute constant $c>0$; see \cite{walfisz}.  One can improve the error term further if one assumes the Riemann hypothesis: see \cite{liu} (or \cite[Problem 969]{EP}) for the most recent results in this direction.

Now, for each positive integer $n$, the translate $\SF - n$ denotes those integers $a$ for which $n+a$ is squarefree.  We recall the following four definitions from \cite{ER81}, and add three more that are implicitly used there.

\begin{definition}[Properties $P$, $Q$, $\overline{P}$, $\overline{P}_\infty$, etc.]
Let $A$ be a set or sequence of natural numbers. If we consider $A$ (which can be finite or infinite) to be a sequence, we always assume it to be strictly increasing.  
\begin{itemize}
    \item We say that $A$ has property $P$ if $\SF - n$ has finite intersection with $A$ for every $n \in \N$. That is, for all $n$, one has $n+a$ squarefree for only finitely many $a \in A$.
    \item We say that $A$ has property $Q$ if $\SF - n$ contains $A \cap [n-1]$ for infinitely many $n \in \N$. That is, for infinitely many $n$, one has $n+a$ squarefree for all $a \in A$ with $a<n$.
    \item We say that $A$ has property $\overline{P}$ if $\SF - n$ contains $A$ for infinitely many $n \in \N$. That is, for infinitely many $n$, one has $n+a$ squarefree for all $a \in A$.
    \item We say that $A$ has property $\overline{P}_\infty$ if $A \setminus (\SF-n)$ is finite for infinitely many $n \in \N$. That is, for infinitely many $n$, one has $n+a$ squarefree for all but finitely many $a \in A$.
    \item We say that $A$ has \emph{squarefree sums} if $\SF - a$ contains $A$ for all $a \in A$. That is, $a+a' \in \SF$ for all $a,a' \in A$.
    \item We say that $A$ is \emph{admissible} if $A$ avoids at least one residue class $b_p \pmod{p^2}$ for each prime $p$.  
    \item We say that $A$ is \emph{almost admissible} if for each prime $p$ there is a residue class $b_p \pmod{p^2}$ that intersects only finitely many elements of $A$. 
\end{itemize}
\end{definition}

We remark that all of these properties are downwardly monotone; whenever a set $A$ has any of the above properties, then so does any subset of $A$. The property $P$, being additionally closed under finite unions, in fact forms an order ideal. Furthermore, we have the following easy implications (summarized in Figures \ref{implications}, \ref{implications-fin}), most of which are already implicit in \cite{ER81}:
\begin{itemize}
    \item Every infinite set with squarefree sums has property $\overline{P}$.
    \item Every set with property $\overline{P}$ has properties $\overline{P}_{\infty}$ and $Q$ as well.
    \item No infinite set with property $P$ obeys $\overline{P}_\infty$ (let alone $\overline{P}$).
    \item Every admissible set is almost admissible.
    \item Every set with property $Q$ is admissible. 
    \item Every set with property $\overline{P}_\infty$ is almost admissible.
    \item Properties $P$ and $\overline{P}_\infty$, as well as the almost admissible property, are unaffected by the addition or deletion of a finite number of elements.  In particular, finite sets obey all three properties.
    \item Finite sets obey $\overline{P}$ (hence $Q$) if and only if they are admissible.
\end{itemize}

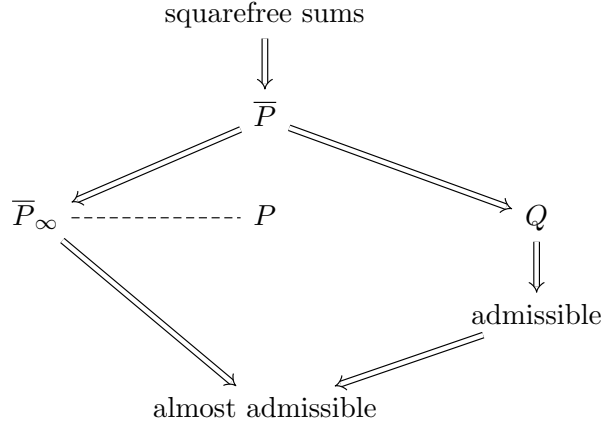
\begin{figure}
\begin{center}
\begin{tikzcd}
    &\text{squarefree sums} \arrow[d, Rightarrow]  \\ 
    &\overline{P} \arrow[dl, Rightarrow] \arrow[dr, Rightarrow]  \\
    \overline{P}_\infty  \arrow[ddr, Rightarrow] \arrow[r, dash, dashed]&P & Q \arrow[d, Rightarrow]\\
     && \text{admissible} \arrow[dl, Rightarrow] \\
    &\text{almost admissible}
\end{tikzcd}
\end{center}
\caption{Implications between various properties for infinite sets $A$. The dashed line indicates the incompatibility of $P$ and $\overline{P}_\infty$ in the infinite case.  Note how $P$ stands in opposition to the other six properties, which are tied to each other by various implications.} \label{implications}
\end{figure}

\begin{figure}
\begin{center}
\begin{tikzcd}
    \text{squarefree sums} \arrow[d, Rightarrow]\\
    \text{admissible} \arrow[d, Rightarrow] \arrow[r, Leftrightarrow] & Q\arrow[r, Leftrightarrow] & \overline{P}  \\
    \text{true} \arrow[r, Leftrightarrow] & \text{almost admissible} \arrow[r, Leftrightarrow] & P \arrow[r, Leftrightarrow] &  \overline{P}_\infty
\end{tikzcd}
\end{center}
\caption{Implications and equivalences between various properties for finite sets $A$.} \label{implications-fin}
\end{figure}
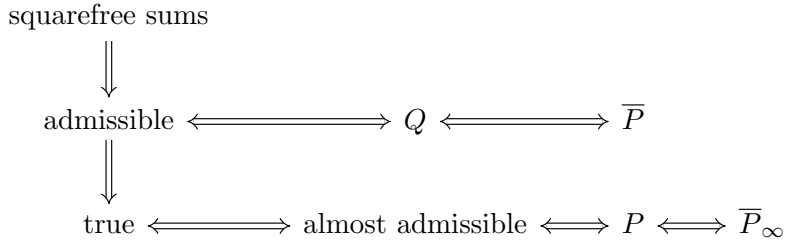

Erd\H{o}s also listed the specific sequences
\begin{align*}
    A_1 &\coloneqq \{ 2^j + 1 : j \in \N \} = \{ 3, 5, 9, 17, 33, \dots\} \quad (\href{https://oeis.org/A000051}{\text{OEIS A000051}})\\
    A_2 &\coloneqq \{ 2^j - 1 : j \in \N \} = \{ 1, 3, 7, 15, 31, \dots\} \quad (\href{https://oeis.org/A000225}{\text{OEIS A000225}})\\
    A_3 &\coloneqq \{ j! + 1 : j \in \N \} = \{ 2, 3, 7, 25, 121, \dots\} \quad (\href{https://oeis.org/A038507}{\text{OEIS A038507}}) \\
    A_4 &\coloneqq \{ j! - 1 : j > 1 \} = \{ 1, 5, 23, 119, 719, \dots\} \quad (\href{https://oeis.org/A033312}{\text{OEIS A033312}})
\end{align*}
as test cases for the above properties.

\subsection{Property \texorpdfstring{$P$}{P}}

Regarding property $P$, Erd\H{o}s writes
\begin{quote}
\emph{Probably a sequence having property $P$ must increase fairly fast, but I have no results in this direction.} \cite[p.\,179]{ER81}
\end{quote}

Contrary to Erd\H{o}s' expectations, sequences with Property $P$ do not necessarily have to grow very quickly.  In Section \ref{P-sec} we show 

\begin{theorem} \label{thmP1}\
\begin{itemize}
    \item[(i)] If $A = \{a_1 < a_2 < \ldots\}$ has property $P$, then $A$ has natural density $0$. That is, the ratio $\frac{a_j}{j}$ goes to infinity with $j$.
    \item[(ii)] Conversely, for any function $f \colon \N \rightarrow \N$ that goes to infinity, there exists an infinite sequence $\{a_1 < a_2 < \ldots\}$ with property $P$ for which $\frac{a_j}{j} \le f(j)$ holds for all $j \in \mathbb{N}$.
\end{itemize}
\end{theorem}

\subsection{Property \texorpdfstring{$Q$}{Q}}

By standard sieve theory and \eqref{basel}, every admissible (or almost admissible) sequence has upper density at most $\frac{6}{\pi^2}$.  Since sequences with property $Q$ are automatically admissible, we conclude

\begin{theorem}[Upper bound]\label{thmQ1-upper}
Every sequence with property $Q$ has upper density at most $\frac{6}{\pi^2}$. 
\end{theorem}

Erd\H{o}s writes

\begin{quote}
\emph{It is easy to see that if $A$ increases sufficiently fast [and is admissible] then it has property $Q$ and in fact there is an $n$, $a_k < n < a_{k+1}$\footnote{Erd\H{o}s actually writes $n_{7k+1}$ instead of $a_{k+1}$, but we assume that this is simply a typographical error.} for which $n+a_i$, $i=1,\dots,k$, is always squarefree. I have no precise information about the rate of increase a sequence having property $Q$ must have.} \cite[p.\,179]{ER81}
\end{quote}

Perhaps surprisingly, the bound in \Cref{thmQ1-upper} is sharp; in Section \ref{Q-sec} we will show

\begin{theorem}[Lower bound]\label{thmQ1-lower}
There exists an infinite sequence $A = \{a_1 < a_2 < \dots\} \subset \SF$ with property $Q$ which has natural density $\frac{6}{\pi^2}$.  Equivalently, one has $\frac{a_j}{j} \to \frac{\pi^2}{6}$ as $j \to \infty$.
\end{theorem}

In Section \ref{suff-sec} we also quantify what counts as ``sufficiently fast'' in Erd\H{o}s' claim:

\begin{theorem}[Sufficiently fast sequences]\label{suff}  There is an absolute constant $C$ such that, if $A = \{a_1 < a_2 < \dots\}$ is an admissible sequence with $a_j \geq \exp(C j / \log j)$ for infinitely many $j$, then $A$ has property $Q$.
\end{theorem}

In fact, as we will explain in Section \ref{explicit}, one can take $C$ to be any constant larger than $4$. If one then furthermore adds a regularity condition on the growth of $A$, the further claim by Erd\H{o}s that suitable $n$ exist between consecutive elements of $A$ also follows.

\begin{theorem}[Sufficiently fast sequences, alternate version] \label{suff-reg} If $A = \{a_1 < a_2 < \dots\}$ is an admissible sequence with $a_j \geq \max\big(\exp(5j / \log j), a_{j-1} + a_{j-1}^{10/11}\big)$ for all sufficiently large $j$, then for all sufficiently large $j$ there exists an $n \in \mathbb{N}$ with $a_{j-1} < n < a_j$ such that $n + a_i \in \SF$ for all $i < j$. 
\end{theorem}

In the converse direction, we can find rapidly growing admissible sequences that do \emph{not} obey property $Q$:

\begin{theorem}[Fast growing counterexample]\label{counter}  There exists an absolute constant $c > 0$ and an admissible sequence $A = \{a_1 < a_2 < \dots\} \subset \SF$ which does not obey property $Q$, and for which the inequality $a_j \geq \exp(c j^{1/2} / \log^{1/2} j)$ holds for all sufficiently large $j$.
\end{theorem}

We establish this result in Section \ref{counter-sec}.

\subsection{Squarefree sums}

As observed in \cite[p. 179]{ER81}, a straightforward greedy construction gives a sequence $A = \{a_1 < a_2 < \dots\}$ with squarefree sums. Erd\H{o}s writes, 

\begin{quote}
\emph{In fact one can find a sequence which grows exponentially. Must such a sequence really increase so fast? I do not expect that there is such a sequence of polynomial growth.} \cite[p.\,179]{ER81}
\end{quote}

Let $\mathrm{ES}_N$ denote the size of the largest subset of $[N]$ with squarefree sums. Erd\H{o}s and S\'ark\"ozy \cite{erdos-sarkozy}, S\'ark\"ozy \cite{sarkozy}, Gyarmati \cite{gyarmati} and Konyagin \cite{konyagin} have given upper and lower bounds on this quantity, with the current best bounds of
$$ \log^2 N \log\log N \ll \mathrm{ES}_N \ll N^{11/15} \exp\left( O\left(\frac{\log N}{\sqrt{\log\log N}} \right)\right)$$
for all large $N$ established in \cite{konyagin}.  (For comparison, the large sieve inequality in Theorem \ref{sqsieve} only gives the weaker upper bound $\mathrm{ES}_N \ll N^{3/4}$.)  Inverting the upper bound, we obtain a lower bound
$$ a_j \gg j^{15/11}\exp\left( -O\left(\frac{\log j}{\sqrt{\log\log j}} \right)\right)$$
for all large $j$ and any infinite sequence $A = \{a_1 < a_2 < \dots\}$ with squarefree sums.  It is likely that the constructions in \cite{konyagin} can be modified to generate an infinite sequence $A = \{a_1 < a_2 < \dots\}$ with squarefree sums such that
$$ \log^2 N \log\log N \ll |A \cap [N]|$$
for all large $N$, although we will not do so here.  Assuming that this is the case, such a sequence would obey the upper bound
$$ a_j \ll \exp( O( j^{1/2} / \log^{1/2} j) ).$$
We remark that a modification of the construction in Theorem \ref{suff} gives an infinite sequence $A = \{a_1 < a_2 < \dots\}$ for which the weaker upper bound
$$ a_j \ll \exp( O(j / \log j) )$$
holds. We refer the reader to the first version of this paper\footnote{\url{https://arxiv.org/abs/2512.01087v1}} for details.

We also note that the opposite problem of studying sets with very few (or no) squarefree pairwise sums was considered in \cite{nathanson}, \cite{filaseta}, \cite{lacampagne}, \cite{schoen}.

\begin{remark}
Erd\H{o}s also writes

\begin{quote}
\emph{Is there a sequence of integers $1 \leq a_1 < a_2 < \dots$ so that for every $i$, $a_i \equiv t \pmod{p^2}$ implies $1 \leq t < p^2/2$? If such a sequence exists then clearly $a_i+a_j$ is always squarefree, but I am doubtful if such a sequence exists.  I formulated this problem while writing these lines and must ask the indulgence of the reader if it turns out to be trivial.} \cite[p.\,179-180]{ER81}
\end{quote}

Indeed, as noted in \cite[Problem 1103]{EP}, it is easy to rule out such a sequence: by the prime number theorem, for $i$ large enough there must exist a prime $p$ between $\sqrt{a_i}$ and $\sqrt{2a_i}$, and it is clear for such a prime that the residue $t = a_i$ of $a_i \pmod{p^2}$ will be larger than $p^2/2$.
\end{remark}

\subsection{Properties \texorpdfstring{$\overline{P}$}{P} and \texorpdfstring{$\overline{P}_\infty$}{P\_infty}}
In contrast to the previous results on squarefree sums, sequences with properties $\overline{P}$ and $\overline{P}_\infty$ can be far denser:

\begin{theorem}\label{overp}\  
\begin{itemize}
    \item[(i)] If $A$ has property $\overline{P}$ or $\overline{P}_\infty$, then the upper density of $A$ is strictly less than $\frac{6}{\pi^2}$.
    \item[(ii)] Conversely, for any $\eps>0$, one can find a sequence $A$ obeying property $\overline{P}$ (and hence $\overline{P}_\infty$) whose lower density (and hence upper density) is at least $\frac{6}{\pi^2}-\eps$.
\end{itemize}
\end{theorem}

We prove this result in Section \ref{overp-sec}.  We summarize some of the above results in Table \ref{table-results}.

\begin{table}
\begin{center}
\begin{tabular}{|c|c|c|}
\hline
Property & Admissibility & Density \\
\hline
$P$ & arbitrary & $0$ (arbitrarily slowly) \\
$Q$ & admissible & $\leq \frac{6}{\pi^2}$ (equality attainable) \\
$\overline{P}$ & admissible & $<\frac{6}{\pi^2}$ (equality almost attainable) \\
$\overline{P}_\infty$ & almost admissible & $<\frac{6}{\pi^2}$ (equality almost attainable) \\
squarefree sums & admissible & $0$ (with at least $O(x^{-1/4})$ decay) \\
\hline
\end{tabular}
\end{center}
\caption{The admissibility and possible densities of (infinite) sequences with the various properties identified by Erd\H{o}s.}\label{table-results}
\end{table}

\subsection{The size of admissible sets}

In \cite[p. 180]{ER81}, Erd\H{o}s defines $A(x)$ (\href{https://oeis.org/A083544}{OEIS A083544}) to be the maximal value of $|A \cap [x]|$ where $A$ ranges over all admissible sets.  By the Chinese remainder theorem and standard sieve theory, $A(x)$ is also the maximal value of $\sum_{y \leq n < y+x} \mu^2(n)$ where $\mu$ is the M\"obius function and where $y$ ranges over all positive integers.  From \eqref{sfx} one has the lower bound
$$A(x) \geq |\SF \cap [x]| = \frac{6}{\pi^2} x + O\left(\frac{\sqrt{x}}{\exp(c \log^{3/5} x / (\log\log x)^{1/5} )}\right)$$
for $x \geq 3$. Erd\H{o}s states that Ruzsa proved that the strict inequality $A(x) > |\SF \cap [x]|$ holds for infinitely many $x$ and writes

\begin{quote}
\emph{Probably this holds for all large $x$. It would be of some interest to estimate $A(x)$ as accurately as possible.} \cite[p.\,180]{ER81}
\end{quote}

By improving upon the aforementioned lower bound on $A(x)$, we can confirm Erd\H{o}s' conjecture.

\begin{theorem}\label{admis-prop} For sufficiently large $x$, one has $\frac{\sqrt{x}}{\log x} \ll A(x) - \frac{6}{\pi^2} x \ll x^{4/5}$.  In particular, we have $A(x) > |\SF \cap [x]|$ for all large $x$.
\end{theorem}

We establish this result in Section \ref{admis-sec}.
By comparing the OEIS sequences \href{https://oeis.org/A013928}{A013928} and \href{https://oeis.org/A083544}{A083544} (see Figure \ref{fig:shift}) it in fact seems likely that $A(x) > |\SF \cap [x]|$ holds for all $x \geq 18$.

\begin{figure}
    \centering
\centerline{\includegraphics[width=\linewidth]{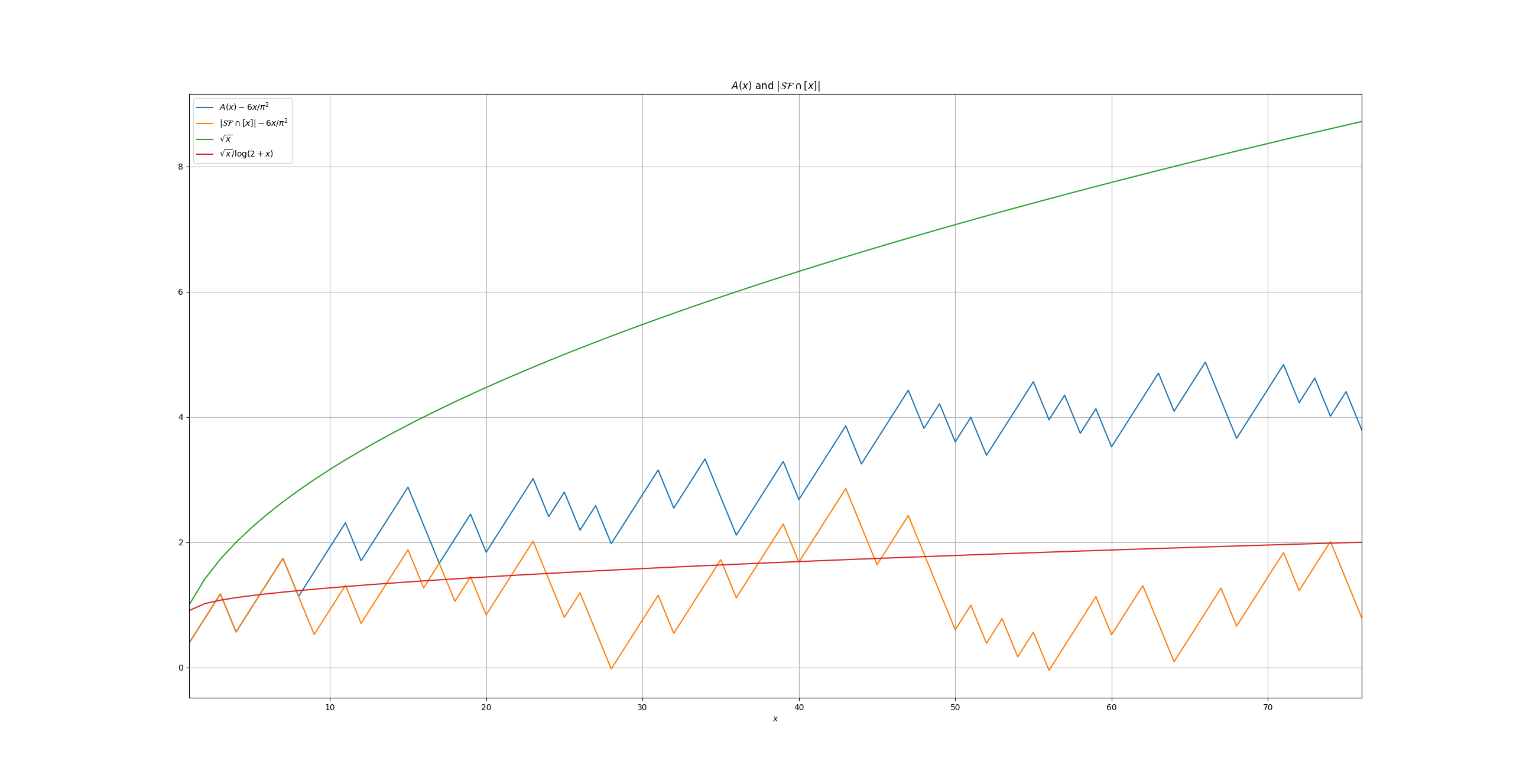}}
    \caption{A comparison of $A(x)$ (\href{https://oeis.org/A083544}{OEIS A083544}) and $|\SF \cap [x]|$ (\href{https://oeis.org/A013928}{OEIS A013928}), after subtracting off the main term $\frac{6}{\pi^2}x$.  All data is drawn from the OEIS.  The numerical range is too short to reach any definitive prediction about the growth rate of
    $A(x) - \frac{6}{\pi^2} x$, although the upper bound of $O(x^{4/5})$ is likely quite far from the truth.}
    \label{fig:shift}
\end{figure}

\subsection{Specific sequences}

We can specialize the above discussion to the four specific sequences $A_1,A_2,A_3,A_4$ singled out in \cite{ER81}.  We first of all claim that all four sequences are admissible (and thus almost admissible).  Indeed, $A_1$ avoids $0 \pmod{4}$ and $1 \pmod{p^2}$ for all odd primes $p$, and $A_2$ similarly avoids $2 \pmod{4}$ and $-1 \pmod{p^2}$ for all odd primes $p$. Further, both $A_3$ and $A_4$ avoid $0 \pmod{4}$, while the elements $j! \pm 1$ are congruent to $\pm 1 \pmod{p^2}$ for all $j \geq 2p$.  Thus, for odd $p$, $A_3$ and $A_4$ occupy at most $2p < p^2$ residue classes modulo $p^2$ and must therefore avoid at least one, demonstrating admissibility.

Since all four sequences grow faster than $\exp(C j / \log j)$, we conclude that $A_1,A_2,A_3,A_4$ all obey property $Q$ by Theorem \ref{suff}.  In fact, the slightly stronger conclusion of Theorem \ref{suff-reg} also holds for these sequences.  On the other hand, the specific sums $3+5$, $1+3$, $3+25$, and $1+23$ are all divisible by $2^2$, showing that none of $A_1,A_2,A_3,A_4$ have squarefree sums.

As for properties $\overline{P}$ and $\overline{P}_\infty$, Artin's primitive root conjecture implies that $2$ is a primitive root of $\Z/p\Z$ for infinitely many primes $p$ (\href{https://oeis.org/A001122}{OEIS A001122}).  Under the additional assumption of the plausible conjecture that infinitely many of these primes are not Wieferich primes\footnote{While it is believed that infinitely many Wieferich primes exist, they are exceedingly rare in practice; the only known examples are $1093$ and $3511$ (\href{https://oeis.org/A001220}{OEIS A001220}).} (i.e., $p^2 \nmid 2^{p-1}-1$), we conclude that $2$ is a primitive root of $\Z/p^2\Z$ for infinitely many $p$.  Assuming the latter, let $n \in \mathbb{N}$ be given, and let $p > n+1$ be a prime for which $2$ is a primitive root of $\Z/p^2\Z$. We then obtain infinitely many $a$ in both $A_1$ and $A_2$ for which $a+n$ is divisible by $p^2$, so that $A_1, A_2$ do not obey properties $\overline{P}$ or $\overline{P}_\infty$.  It looks difficult to establish this fact unconditionally, however.  For $A_3, A_4$, the situation seems delicate: for fixed squarefree $n$, the expression $j!+n$ is, for all $j \ge 2n$, not divisible by $p^2$ for any $p \leq j$, so by \eqref{psum} below one heuristically predicts this number to be squarefree with a ``probability'' of $1 - O(1/j \log j)$.  The sum $\sum_j \frac{1}{j \log j}$ is barely divergent, suggesting that there are infinitely many $j$ for which $j! + n$ is not squarefree, implying that $A_3$ and $A_4$ do not obey the properties $\overline{P}, \overline{P}_\infty$.  Given the special structure of the sequence $j!+n$ however, we do not find this to be a completely convincing argument.

Probabilistic heuristics also suggest that $A_1, A_2, A_3, A_4$ do not have property $P$, but this looks difficult to establish with known techniques.  Even just verifying that (say) $2^j+1$ can be squarefree for infinitely many $j$, while extremely plausible from such heuristics, does not seem amenable to existing tools such as sieve theory due to the rapid growth of the sequence, even if one assumes powerful conjectures such as the abc conjecture or Artin's conjecture.  In \cite[p. 178]{ER81} it is written that ``of course it seems hopeless'' to solve these problems.

\subsection{Primes and \texorpdfstring{$k$}{k}-free sequences}

Erd\H{o}s also writes

\begin{quote}
\emph{These problems can of course be stated for other sequences than $p^2$, but we formulate only one such question: Is there an infinite sequence $a_1 < a_2 < \dots$ so that there are infinitely many $n$ for which for all $a_k < n$, $n+a_k$ is always a prime?} \cite[p.\,179]{ER81}
\end{quote}

It turns out that with the recent result \cite{taoziegler}, this question is answered in the affirmative.  Indeed, from \cite[Corollary 1.6]{taoziegler} one can find $a_1 < a_2 < \dots$ and $b_1 < b_2 < \dots$ such that $a_i + b_j$ is prime for all $1 \leq i < j$.  Taking a sufficiently sparse subsequence $a_{j_l}$ for which $a_{j_{l+1}} > b_{j_l}$ holds for all $l$, then gives the claim.

Moving on to higher powers instead, without too much difficulty one can extend the results in this paper from squarefree integers to $k$-free integers, for any fixed $k \ge 2$. That is, by making the obvious necessary numerical and definitional changes, we may formulate the results for $\N \backslash \bigcup_p p^k \N$ as opposed to $\SF$. For instance, $\zeta(2) = \pi^2/6$ should then be replaced by $\zeta(k)$ everywhere, suitable numerical adjustments should be made to Theorem \ref{suff-reg}, the exponents $1/2$ in Theorem \ref{counter} are to be replaced by $1/k$ (and $\sqrt{x}$ similarly replaced by $x^{1/k}$ in \eqref{sfx} and Theorem \ref{admis-prop}), and the exponent $4/5$ in Theorem \ref{admis-prop} should be replaced by $\frac{2k}{3k-1}$.  We leave the full proofs of these generalizations to the interested reader. 

\subsection{Notation}\label{notation-sec}

All sums and products over $p$ are understood to be over primes, with $2 = p_1 < p_2 < \ldots$ denoting the sequence of primes.  If $E$ is a finite set, we use $|E|$ to denote its cardinality.  For any $x$, we use $[x]$ to denote the discrete interval
$$ [x] \coloneqq \{ n \in \N: n \leq x \}.$$

If $A$ is a set or sequence of natural numbers, we define its upper and lower density to be
$$ \limsup_{x \to \infty} \frac{|A \cap [x]|}{x}$$
and
$$ \liminf_{x \to \infty} \frac{|A \cap [x]|}{x}$$
respectively. If the upper and lower densities agree, then they are both equal to the natural density, defined as
$$ \lim_{x \to \infty} \frac{|A \cap [x]|}{x}.$$

We use $X \ll Y$, $Y \gg X$, or $X = O(Y)$ to denote the estimate $|X| \leq CY$ for some absolute constant $C$.  If there is a parameter $k$ that goes to infinity, we use $X = o(Y)$ to denote the estimate $|X| \leq c(k) Y$ for some function $c(k)$ of that parameter that goes to zero as $k \to \infty$.  Finally, we write $X \asymp Y$ for $X \ll Y \ll X$.

\subsection{Acknowledgments}

We thank Thomas Bloom for his website \cite{EP} where the above problems were brought to the authors' attention at \cite[Problems 1102, 1103]{EP}.  We also thank Benjamin Bedert for references concerning sets with squarefree sums.

TT was supported by the James and Carol Collins Chair, the Mathematical Analysis \& Application Research Fund, and by NSF grants DMS-2347850, and is particularly grateful to recent donors to the Research Fund. 

\section{Basic estimates}

From the prime number theorem we have
$$ \sum_{x \leq p < 2x} \frac{1}{p^2} \asymp \frac{1}{x \log x}$$
for all large $x$, and thus on summing
\begin{equation}\label{psum}
\sum_{p \geq x} \frac{1}{p^2} \asymp \frac{1}{x \log x} \ll \frac{1}{x}
\end{equation}
for any $x \geq 2$.  Similar arguments give the variant bounds 
\begin{align}
\sum_{p \geq x} \frac{1}{p (\log\log p)^2} &\asymp \frac{1}{\log \log x} \label{plog}\\
\sum_{p \geq x} \frac{\log\log p}{p^2} &\asymp \frac{\log\log x}{x \log x} \label{plog-3}\\
\sum_{2 < p \leq x} \frac{p}{(\log\log p)^2} &\asymp \frac{x^2}{\log x(\log \log x)^2} \label{plog-2} \\
\sum_{2 < p \leq x} \log\log p &\asymp \frac{x \log\log x}{\log x} \label{plog-4}
\end{align}
for any $x \geq 3$.

The following basic probabilistic calculation will be used frequently:

\begin{lemma}\label{freq}  Let $b \pmod{W}$, $c \pmod{q}$ be congruence classes with $W$ coprime to $q$, and let $I$ be an interval of length $L \geq W$.  Then, if $n$ is drawn uniformly at random from those elements of $b \pmod{W}$ that lie in $I$, the probability that $n$ lies in $c \pmod{q}$ is $O( \frac{1}{q} + \frac{W}{L})$.
\end{lemma}

\begin{proof}  Clearly, $n$ ranges in a set of size $\asymp \frac{L}{W}$.  Meanwhile, by the Chinese remainder theorem, the intersection of $b \pmod{W}$ and $c \pmod{q}$ is a residue class modulo $Wq$, which intersects $I$ in $O( \frac{L}{Wq} + 1)$ points.  The claim follows.
\end{proof}

\section{Property \texorpdfstring{$P$}{P}}\label{P-sec}

We now show \Cref{thmP1}.

\begin{proof}
We will first show part (ii). Let $f(j)$ be any function that goes to infinity with $j$ and define $W_r \coloneqq (p_1 p_2 \cdots p_r)^2$ with $W_0 = 1$, where we recall that $p_1 < p_2 < \dots$ is the sequence of primes. Let $0 = l_0 < l_1 < \ldots $ be a strictly increasing sequence of integers for which $f(j) \ge W_r$ holds for all $j \ge l_r$, which exists by the assumption that $f(j)$ goes to infinity. Finally, for a positive integer $j$, let $k(j)$ be such that $l_{k(j)} \le j < l_{k(j)+1}$. Now we inductively define the sequence $A = \{1 = a_1 < a_2 < \ldots \}$ as follows. First off, $a_j = j$ for all $j \le l_1$. Then, assuming that $a_{j-1}$ is defined, we define $a_j$ to be the smallest integer larger than $a_{j-1}$ for which $a_{j} \equiv -r \pmod{p_r^2}$ holds for all $r$ with $1 \le r \le k(j)$. We then need to show that the sequence thusly defined has property $P$, and we need to verify that the inequality $\frac{a_j}{j} \le f(j)$ holds for all $j \in \mathbb{N}$.

By the Chinese remainder theorem, we know that $a_{j'} \le a_{j'-1} + W_{k(j')} \le a_{j'-1} + W_{k(j)}$ holds for all $j' \le j$, so that $a_j \le jW_{k(j)}$ by induction. Since $f(j) \ge W_{k(j)}$, the desired inequality holds.

To see that the sequence we defined has property $P$, let $n$ be any positive integer and let $j$ be any integer larger than or equal to $l_n$. Then $k(j) \ge n$, so that (by using $r = n$ in the definition of $a_j$) we see $a_j \equiv -n \pmod{p_n^2}$ implying $a_j + n \notin \SF$. As $n$ was arbitrary and $j$ was any integer larger than or equal to $l_n$, we conclude that $\SF - n$ has finite intersection with $A$ for all $n \in \mathbb{N}$, so that $A$ has property $P$.

To prove the part (i) of Theorem \ref{thmP1}, assume that $A$ is sequence with positive upper density $\delta$.  Let $C$ be a large absolute constant, and let $H$ be sufficiently large depending on $\delta$ and $C$.  We claim that there exist arbitrarily large $a \in A$ for which $a+[H]$ intersects $\SF$.  By the pigeonhole principle, this implies that an $n \in [H]$ exists such that $n+a \in \SF$ for infinitely many $a$, establishing that $A$ does not have property $P$, as desired.

It remains to establish the claim.  We will use the Maier matrix method, which in this context is based on counting\footnote{One can imagine an $A \times [H]$ matrix with entry $a+h$ in the $a^{\mathrm{th}}$ row and $h^{\mathrm{th}}$ column.  Roughly speaking, we want to sieve $a+h$ to be squarefree for many entries in a column, but initially it is easier to obtain many squarefree (or nearly squarefree) entries in a row.}
the number of pairs $(a,h) \in A \times [H]$ for which $a+h \in \SF$.  Since $A$ has upper density $\delta$, there exist arbitrarily large $N$ such that
$$ |A \cap [N]| \geq \frac{\delta N}{2}$$
and hence
$$ |A \cap [\delta N/4, N]| \geq \frac{\delta N}{4}.$$
In particular we can take $N \geq H$.  For $H$ large enough, we see from \eqref{basel} and standard sieving that for any given $a \in A$, there exist $\asymp H$ choices of $h \in [H]$ such that $p^2 \nmid a+h$ for all $p \leq C$.  By Lemma \ref{freq}, for each $C < p \leq H^{1/2}$, the number of $h \in [H]$ for which $p^2|a+h$ is $O(H/p^2)$.  Applying \eqref{psum}, we conclude for $C$ large enough that there are $\asymp H$ choices of $h \in [H]$ such that $p^2 \nmid a+h$ for all $p \leq H^{1/2}$.  Thus
$$ |\{ (a,h) \in (A \cap [\delta N/4, N]) \times [H]: p^2 \nmid a+h \hbox{ for all } p \leq H^{1/2} \}| \gg \delta N H.$$
On the other hand, for each $p > H^{1/2}$, there are $O(N/p^2)$ possible values of $a+h = O(N)$ that are divisible by $p^2$. Each of these is associated to at most $O(H)$ pairs $(a,h)$, hence by \eqref{psum} we have
$$ |\{ (a,h) \in (A \cap [\delta N/4, N]) \times [H]: a+h \in \SF \}| \gg \delta N H - O\left( \frac{NH}{H^{1/2} \log H} \right).$$
Taking $C$ large enough, we conclude that this set is non-empty, which gives the claim.
\end{proof}

\begin{remark} The above argument shows that the $n$ for which $n+a \in \SF$ for infinitely many $a \in A$ can be chosen to be of size $O( 1 / \delta^2 \log^2(1+\frac{1}{\delta}) )$. The bound can be improved by taking advantage of known moment estimates
$$ \sum_{s_{n+1} \leq x} (s_{n+1} - s_n)^\gamma \leq C_\gamma x$$
for the gaps between consecutive squarefree numbers $s_n, s_{n+1}$ and some exponent $\gamma > 1$.  Any such bound can be used to bound $n$ by $O(1/\delta^{1/(\gamma-1)})$ by applying Markov's inequality.  Such a moment estimate is currently known for $\gamma < 3.75$ \cite{chan}, thus one can take $n \ll 1/\delta^{\frac{1}{2.75}+o(1)}$ in the limit $\delta \to 0$.  
\end{remark}

\section{Property \texorpdfstring{$Q$}{Q}}

\subsection{The lower bound}\label{Q-sec}

In this section we establish \Cref{thmQ1-lower}.
The key proposition used for this construction will be

\begin{proposition}[Key construction]\label{key}  Let $0 < \eps < 1$ and let $n$ be a sufficiently large natural number depending on $\eps$.  Then there exist arbitrarily large natural numbers $n'$ with the following two properties:
\begin{itemize}
    \item[(i)] $n'+a \in \SF$ for all $a \in \SF \cap [n]$.
    \item[(ii)] For any $n \leq R \leq n'$, one has
    $$ \frac{|\{a \in \SF \cap [R]: n'+a \in \SF \}|}{R} \geq \frac{6}{\pi^2} - O(\eps).$$
\end{itemize}
\end{proposition}

Let us assume Proposition \ref{key} for the moment and conclude the proof of Theorem \ref{thmQ1-lower}.  By starting with a sufficiently large $n_1$ and applying the above proposition iteratively with $\eps=1/k$ for $k=1,2,\dots$, one can find a rapidly growing sequence $1 < n_1 < n_2 < \dots$ with the following properties for all natural numbers $k$:
\begin{itemize}
    \item[(a)]  $n_{k+1}+a \in \SF$ whenever $a \in \SF \cap [n_k]$.
    \item[(b)]  For any $n_k \leq R \leq n_{k+1}$, one has
    $$ \frac{|\{a \in \SF \cap [R]: n_{k+1}+a \in \SF \}|}{R} \geq \frac{6}{\pi^2} - O\left(\frac{1}{k}\right).$$
    \item[(c)] One has $n_{k+1} \geq (k+1) n_k$.
\end{itemize}

We then define
\begin{equation}\label{adef}
 A \coloneqq \bigcup_{k=1}^\infty \{ a \in \SF \cap (n_k,n_{k+1}]: n_{k+1}+a \in \SF \}.
 \end{equation}

Clearly $A \subset \SF$.  We claim that $A$ has property $Q$.  Indeed, it will suffice to show that for any $k$, one has $n_{k+1}+a \in \SF$ for all $a < n_{k+1}$.  For $a \leq n_k$ this follows from property (a) (since $A \subset \SF$) and for $n_k < a < n_{k+1}$ this follows by construction from \eqref{adef}.

Since $A \subset \SF$, it follows from both \eqref{sfx} and Theorem \ref{thmQ1-upper} that $A$ has upper density at most $\frac{6}{\pi^2}$.  To obtain the conclusion of Theorem \ref{thmQ1-lower}, it thus suffices to show that $A$ has lower density at least $\frac{6}{\pi^2}$.  It will suffice to show that for all $k \geq 2$ and all $n_k \leq R \leq n_{k+1}$, one has the lower bound
\begin{equation}\label{ara}
 \frac{|A \cap [R]|}{R} \geq \frac{6}{\pi^2} - o(1)
\end{equation}
as $k \to \infty$.
From \eqref{adef} we have
$$ |\{A \cap [R] \}| \geq |\{a \in \SF \cap [R]: n_{k+1}+a \in \SF \}|
- |\{a \in \SF \cap [n_k]: n_{k}+a \notin \SF \}| - n_{k-1}$$
for $k \ge 2$. From property (b) we have
$$ |\{a \in \SF \cap [R]: n_{k+1}+a \in \SF \}| \geq \frac{6}{\pi^2} R - o(R)$$
and
$$ |\{a \in \SF \cap [n_k]: n_k+a \in \SF \}| \geq \frac{6}{\pi^2} n_k - o(n_k)$$
and hence by \eqref{sfx}
$$ |\{a \in \SF \cap [n_k]: n_k+a \notin \SF \}| = o(n_k) = o(R).$$
Also by property (c) we have $n_{k-1} \leq n_k/k = o(R)$.  Combining these bounds we obtain \eqref{ara}.

It remains to establish Proposition \ref{key}.  Let $\eps$ and $n$ be as in the proposition. Set
$$ W \coloneqq \prod_{p \leq n^2} p^2$$
and let $x$ be sufficiently large depending on $\eps$ and $n$.  Let $n'$ be chosen uniformly at random among the multiples of $W$ in $[x/2,x]$.  It will suffice to show that $n'$ obeys properties (i) and (ii) of the proposition with probability $1-o(1)$ as $n \to \infty$.

We begin with (i).  If $a \in \SF$, then for any prime $p \leq n^2$, $p^2$ divides $W$ and hence $n'$, but not $n'+a$.  Also, if $a \leq n$, then $n'+a \leq 2x$. Thus, if $n'+a \notin \SF$, then $n'+a$ must be divisible by $p^2$ for some $n^2 < p \leq \sqrt{2x}$.  On the other hand, from Lemma \ref{freq} we see that for each such $p$ and any $a \leq n$, the probability that $n+a$ is divisible by $p^2$ is $O( \frac{1}{p^2} + \frac{W}{x})$.  By the union bound, the probability that (i) fails is then at most
$$ \ll \sum_{a=1}^n \sum_{n^2 < p \leq \sqrt{2x}} \left(\frac{1}{p^2} + \frac{W}{x}\right)
\ll \frac{1}{n} + \frac{Wn}{\sqrt{x}} \ll \frac{1}{n}$$
thanks to \eqref{psum} and the assumption that $x$ is sufficiently large. The claim (i) follows.

Now we turn to (ii) and we aim to show that the inequality
$$ |\{a \in \SF \cap [R]: n'+a \in \SF \}| \geq \frac{6}{\pi^2} R - O(\eps R)$$
holds for all $n \leq R \leq x$ with high probability.  By the error term of $O(\eps R)$, we may assume without loss of generality that $R$ is a power of $1+\eps$.  And by a telescoping series, it will then suffice to show that with probability $1-o(1)$, one has
$$ |\{a \in \SF \cap (R,(1+\eps) R]: n'+a \in \SF \}| \geq \frac{6}{\pi^2} \eps R - O(\eps^2 R)$$
for all $n \leq R \leq x$ that are powers of $1+\eps$.

So let $R$ be a fixed power of $1+\epsilon$.  From \eqref{sfx}, we have
\begin{equation}\label{rs}
|\SF \cap (R,(1+\eps) R]| = \frac{6}{\pi^2} \eps R + O(\eps^2 R).
 \end{equation}
As before, for any $a \in \SF \cap (R,(1+\eps) R]$, if $n'+a \notin \SF$, then $n'+a$ must be divisible by $p^2$ for some $n^2 < p \leq \sqrt{2x}$.  Each prime $p$ with $n^2 < p \leq \sqrt{R}$ removes at most $O(R/p^2)$ elements from this set, for a total of at most
$$ \ll \sum_{n^2 < p \leq \sqrt{R}} \frac{R}{p^2} \ll \frac{R}{n^2} \ll \eps^2 R$$
elements (this claim is vacuously true if $\sqrt{R} \leq n^2$), thanks to \eqref{psum}.  So it remains to consider the event that for some $a$ in \eqref{rs}
$n'+a$ is divisible by $p^2$ for some $\max(n^2,\sqrt{R}) < p \leq \sqrt{2x}$.  By Lemma \ref{freq} and the union bound, this event holds for a given $R$ with probability
$$ \ll \sum_{a=1}^n \sum_{\max(n^2,\sqrt{R}) < p \leq \sqrt{2x}} \left(\frac{1}{p^2} + \frac{W}{x}\right)
\ll \frac{n}{\max(n^2,\sqrt{R})} + \frac{Wn}{x}$$
thanks to \eqref{psum}. Summing over $n \leq R \leq x$ that are powers of $1+\eps$, the probability that (ii) fails is then
$$
\ll \frac{1}{\eps} \left( \frac{\log n}{n} + \frac{Wn \log x}{x} \right) \ll \frac{\log n}{\eps n}$$
since $x$ is large.  This concludes the proof of Proposition \ref{key} and hence Theorem \ref{thmQ1-lower}.

\subsection{Sufficiently rapidly growing sequences}\label{suff-sec}

Now we adapt the above arguments to prove Theorem \ref{suff}.  Let $c>0$ be a small constant, and let $C$ be sufficiently large depending on $c$.
Suppose that $A = \{a_1 < a_2 < \dots\}$ is an admissible sequence with
\begin{equation}\label{Ajj}
a_j \geq \exp(C j / \log j)
\end{equation}
for infinitely many $j$.

Let $j$ be a sufficiently large integer for which \eqref{Ajj} holds and set $x = a_j$. From \eqref{Ajj} we then see that $j \leq c \log x \log\log x$.  Now, with
$$ W \coloneqq \prod_{p \leq \frac{1}{10} \log x} p^2,$$
we have $W = x^{1/5+o(1)}$ as $x \to \infty$, by the prime number theorem.  Moreover, by the Chinese remainder theorem and admissibility, there exists an integer $b$ such that the sequence $A$ avoids the residue class $b \pmod{p^2}$ for all $p | W$.  Fix such a $b$, and let $n$ be chosen uniformly at random from the elements of the residue class $-b \pmod{W}$  in $[x/2,x]$.  We then aim to show that there is a positive probability that $n+a \in \SF$ for all $a \in A \cap [n] \subseteq A \cap [x]$. 

By construction, $n+a$ is not divisible by $p^2$ for any $p \leq \frac{1}{10} \log x$ and $a \in A$. For $a \leq x$ we get $n+a \leq 2x$, and such numbers are also not divisible by $p^2$ for any $p > \sqrt{2x}$.  So the only remaining primes $p$ to check are those with $\frac{1}{10} \log x < p \leq \sqrt{2x}$. By Lemma \ref{freq} and the union bound, the probability that $n+a$ is not in $\SF$ for some $a \in A \cap [x]$, is at most
$$ \ll j \sum_{\frac{1}{10} \log x < p \leq \sqrt{2x}} \left( \frac{1}{p^2} + \frac{W}{x}\right)$$
which by \eqref{psum}, $W = x^{1/5+o(1)}$ and the inequality $j \leq c \log x \log \log x$ can be bounded by
$$ \ll (c \log x \log\log x) \left( \frac{1}{\log x\log\log x} + \frac{x^{7/10+o(1)}}{x} \right) \ll c.$$
For $c$ small enough and $j$ (and therefore $x$) large enough, this failure probability is smaller than $1$, hence an $n \in [x/2,x]$ exists for which $n+a \in \SF$ for all $a \in A$ with $a<n$.  Since by assumption there are infinitely many $j$ for which \eqref{Ajj} holds, we see that $A$ has property $Q$ as claimed.

\subsection{Explicit version} \label{explicit}
In Section \ref{intro} we claimed that one can take $C$ in Theorem \ref{suff} to be any constant larger than $4$. To see this, we first note that we can improve \eqref{psum} and the upper bound probability in Lemma \ref{freq} to $\sum_{p \geq x} \frac{1}{p^2} = \frac{1 + o(1)}{x \log x}$ and $\frac{1}{q} + O(\frac{W}{L})$ respectively. Secondly, we remark that it is possible to increase the constant $\frac{1}{10}$ that appears in the definition of $W$, to $\frac{1}{4} - \epsilon$. With $c = \frac{1}{4} - O(\epsilon)$ and $C = \frac{1}{c} + O(\epsilon)$, one can then check that the final calculation to upper bound the failure probability (which currently gives $\ll c$) instead gives a quantity that is, for sufficiently large $x$, smaller than $1$.

Furthermore, recall that in the proof of Theorem \ref{suff} we chose $n$ in the interval $[x/2, x]$. If we instead use the interval $(x, x + x^{10/11})$, apply the aforementioned slightly improved versions of \eqref{psum} and Lemma \ref{freq}, and use the assumption that $a_j \geq \max\big(\exp(5j / \log j), a_{j-1} + a_{j-1}^{10/11}\big)$ holds for all sufficiently large $j$, Theorem \ref{suff-reg} follows analogously.

\subsection{A fast growing admissible sequence that does not obey \texorpdfstring{$Q$}{Q}}\label{counter-sec}

We now show Theorem \ref{counter}.  With $C > 0$ a large constant, we construct a random subset $A = \{a_1 < a_2 < \dots \}$ of $\SF$ by having each squarefree $n \geq 3$ independently lie in $A$ with probability $\mu_n \coloneqq \min( C \frac{\log n \log\log n}{n}, 1)$.  As $A$ is a subset of $\SF$, it is admissible.  It will then suffice to show the following claims:
\begin{itemize}
    \item[(i)] Almost surely, one has $a_j \geq \exp( c j^{1/2} / \log^{1/2} j )$ for all sufficiently large $j$ and some constant $c>0$ depending on $C$.
    \item[(ii)]  Almost surely, for sufficiently large $x$ that are powers of two, there does not exist any $n \in [x,2x]$ such that $n+a$ is squarefree for all $a \in A \cap [x]$.  In particular, $A$ does not obey property $Q$.
\end{itemize}

We begin with (i).  For sufficiently large $x$, the cardinality $|A \cap [x]|$ is a sum of independent Bernoulli variables, with mean 
$$\sum_{n \in \SF \cap [x]} \mu_n \asymp (C + o(1)) \log^2 x \log\log x$$
and variance
$$\sum_{n \in \SF \cap [x]} \mu_n (1-\mu_n) \asymp (C + o(1)) \log^2 x \log\log x.$$
Applying Bennett's inequality\footnote{One could also use Bernstein's inequality or Azuma's inequality here if desired.} \cite{bennett}, we conclude that
$$ \mathbf{P}( |A \cap [x]| \geq C^2 \log^2 x \log\log x ) \le \exp( - C \log^2 x \log\log x ),$$
if $C$ is large enough. Applying the Borel--Cantelli lemma, we conclude that almost surely, one has 
$$ |A \cap [x]| <  C^2 \log^2 x \log\log x$$ 
for all sufficiently large $x$. In other words, if $x$ is sufficiently large, then we almost surely have $a_j > x$ whenever $j > C^2 \log^2 x \log\log x$.  The claim (i) then follows by a routine calculation.

Now we turn to (ii).  Let $x$ be a sufficiently large power of two and suppose that $A \cap [x]$ occupies every non-zero residue class modulo $p^2$ for each prime $p \leq \log x$.  Then, if $n$ has the property that $n+a$ is squarefree for all $a \in A \cap [x]$, it is necessary for $n$ to be congruent to $0 \pmod{p^2}$ for all $p \leq \log x$. In particular, $n$ must be divisible by $\prod_{p \leq \log x} p^2$, so that no element $n \in [x,2x]$ can have this property by the prime number theorem.

It will therefore suffice to prove that, almost surely, $A \cap [x]$ occupies every non-zero residue class modulo $p^2$ for each prime $p \leq \log x$, for all sufficiently large powers of two $x$.  Any given non-zero residue class $b \pmod{p^2}$ avoids $A \cap [x]$ with probability
\begin{align*}
\prod_{\substack{n \in [x] \\ n \in \SF \cap b \hspace{-8pt}\pmod{p^2}}} (1 - \mu_n) &<
\prod_{\substack{\sqrt{x} \leq n \leq x \\ n \in \SF \cap b \hspace{-8pt}\pmod{p^2}}} \exp(- \mu_n) \\
&= \exp\left(-\sum_{\substack{\sqrt{x} \leq n \leq x \\ n \in \SF \cap b \hspace{-8pt}\pmod{p^2}}} \frac{C \log n \log\log n}{n}\right) \\
&\le \exp\left(-\sum_{\substack{\sqrt{x} \leq n \leq x \\ n \in \SF \cap b \hspace{-8pt}\pmod{p^2}}} \frac{C \log x \log\log x}{4n}\right). 
\end{align*}
By standard sieving, the number of elements of $\SF \cap b \pmod{p^2}$ in $[y,2y]$ is $\gg y/p^2$ for any $\sqrt{x} \leq y \leq x$. We can therefore upper bound the above probability by $\exp( -5 \log^2 x \log\log x / p^2 ) \leq \log^{-5} x$, by setting $C$ large enough.  Since the number of non-zero residue classes $b \pmod{p^2}$ is smaller than $p^2 \le \log^2 x$, while the number of primes $p$ is smaller than $\log x$, by the union bound we conclude that $A \cap [x]$ fails to occupy every non-zero residue class modulo $p^2$ for each prime $p \leq \log x$ with probability smaller than $\log^{-2} x$.  By summing the failure probability over all sufficiently large powers of two, claim (ii) follows from the Borel--Cantelli lemma.

\section{Property \texorpdfstring{$\overline{P}$}{P}}\label{overp-sec}

We now prove Theorem \ref{overp}, starting with part (i).  If $A$ obeys either property $\overline{P}$ or $\overline{P}_\infty$, then there exist $n_1 < n_2$ such that $n_1+a, n_2+a \in \SF$ for all but finitely many $a \in A$.  Thus, after deleting finitely many elements of $A$, the remaining elements avoid at least one residue class modulo $p^2$ for all primes $p$, and at least two residue classes modulo $p^2$ for all $p > n_2 - n_1$. From standard sieve bounds we conclude that the upper density of $A$ is strictly less than \eqref{basel}, giving (i).  In fact, by \ref{psum} we obtain a more precise upper bound $\frac{6}{\pi^2} - \frac{c}{(n_2-n_1) \log(n_2-n_1+1)}$ for this density, for some absolute constant $c>0$.

Now we prove (ii).  We need the following lemma:

\begin{lemma} \label{largeP} For any sufficiently large $P \ge 3$, there exist arbitrarily large natural numbers $n$ such that
\begin{itemize}
    \item[(a)] $n \equiv 0 \pmod{p^2}$ whenever $p \leq P$; and
    \item[(b)] $n + a \not \equiv 0 \pmod{p^2}$ whenever $p>P$ and $1 \leq a \leq \frac{p}{(\log\log p)^2}$.
\end{itemize}
\end{lemma}

\begin{proof}  Write $W \coloneqq \prod_{p \leq P} p^2$, and let $x$ be sufficiently large depending on $W$.  We select $n$ uniformly at random from the elements of $[x/2,x]$ that are congruent to $0 \pmod{W}$.  Clearly (a) holds.  Since $n \leq x$, the claim (b) is automatic for $p \geq \sqrt{2x}$ since $n+a < p^2$ for such $p$, for any $a \leq \frac{p}{(\log\log p)^2}$.  By Lemma \ref{freq}, the probability that $p^2$ divides $n+a$ for a given $P < p \leq \sqrt{2x}$ and $a$ is $O(1/p^2 + W/x)$, thus the total failure probability of (b) is
$$ \ll \sum_{P < p \leq \sqrt{2x}} \left(\frac{1}{p (\log\log p)^2} + \frac{W p}{x (\log\log p)^2}\right),$$
which by \eqref{plog}, \eqref{plog-2} can be bounded by
$$ \ll \frac{1}{\log\log P} + \frac{W}{\log x (\log\log x)^2}.$$
For $P$ and $x$ large enough this will be smaller than $1$, giving the claim.
\end{proof}

With $K \geq 3$ sufficiently large (but otherwise arbitrary) we can, by Lemma \ref{largeP}, find an infinite sequence $n_1 < n_2 < \dots$ such that
\begin{itemize}
    \item[(a)] $n_j \equiv 0 \pmod{p^2}$ whenever $p \leq K \exp\exp j$; and
    \item[(b)] $n_j + a \not \equiv 0 \pmod{p^2}$ whenever $p> K\exp\exp j$ and $1 \leq a \leq \frac{p}{(\log\log p)^2}$.
\end{itemize}
Now set $A$ to be the set of all natural numbers $a$ such that $n_j+a \in \SF$ for all $j$.  This set obeys properties $\overline{P}$ and $\overline{P}_\infty$ by definition, so it suffices to show that
$$ |A \cap [x]| \geq \frac{6}{\pi^2}x - o(x) - O\left(\frac{x \log\log K}{K \log K}\right).$$
Indeed, as $K$ is arbitrary one can make $O\left(\frac{x \log\log K}{K \log K}\right)$ smaller than $\epsilon x$ by choosing $K$ sufficiently large.

Suppose that $a \in \SF \cap [x]$ fails to lie in $A$ for some $x \ge 3$. In other words, there must be a prime $p$ and an index $j$ such that $p^2$ divides $n_j+a$.  By property (a) and the squarefree nature of $a$, this forces
$$
p > K \exp\exp j,
$$
so in particular $p > K \ge 3$ and
$$ j < \log\log p.$$
By property (b), this then forces
$$ x \geq a > \frac{p}{(\log\log p)^2}$$
and thus
$$ p \ll x (\log\log x)^2.$$
We conclude that $A$ contains all elements of $\SF \cap [x]$, after removing fewer than $\log\log p$ residue classes modulo $p^2$ for all $K < p \ll x(\log\log x)^2$.  By \eqref{sfx} we get a lower bound of
$$ |A \cap [x]| \geq \frac{6}{\pi^2}x - o(x) - \sum_{K < p \ll x (\log\log x)^2} O\left( \frac{x \log\log p}{p^2} + \log\log p\right)$$
which by \eqref{plog-3} and \eqref{plog-4} simplifies to
\begin{align*}
|A \cap [x]| &\geq \frac{6}{\pi^2}x - o(x) - O\left( \frac{x \log\log K}{K \log K} + \frac{x (\log\log x)^3}{\log x} \right) \\
&= \frac{6}{\pi^2}x  - o(x) - O\left( \frac{x \log\log K}{K \log K}\right),
\end{align*}
finishing the proof.

\section{Size of admissible sets}\label{admis-sec}

In this section we prove Theorem \ref{admis-prop}, beginning with the upper bound.  Applying Theorem \ref{sqsieve} with $\omega(p^2) = 1$ for all $p$, and we get
$$ A(x) \leq \frac{x + Q^4}{\sum_{q \leq Q} h(q)}$$
for any $x, Q \geq 1$, where $h(q) = \mu^2(q) \prod_{p|q} \frac{1}{p^2-1}$.  Since $h(q) = O(1/q^2)$ and 
$$\sum_{q \ge 1} h(q) = \prod_{p \ge 2} \left(1 + \frac{1}{p^2-1}\right) = \zeta(2) = \frac{\pi^2}{6},$$
we have $\sum_{q \leq Q} h(q) = \frac{\pi^2}{6} - O(1/Q)$, and hence
$$ A(x) \leq \frac{6}{\pi^2} x + O\left(\frac{x}{Q} + Q^4\right).$$
Setting $Q \coloneqq x^{1/5}$, we obtain the claim.

Now we turn to the lower bound.  Let $W$ be the product of all primes $p \leq \sqrt{x}$, and let $n$ be chosen uniformly at random from $\Z/W\Z$.  For any shift $a \in \Z$, the probability that $n+a$ is not divisible by any $p^2$ with $p \leq \sqrt{x}$ is exactly
$$ \prod_{p \leq \sqrt{x}} \left(1 - \frac{1}{p^2}\right),$$
thanks to the Chinese remainder theorem.  In particular, by \eqref{basel} and \eqref{psum}, this probability is at least $\frac{6}{\pi^2} + \frac{c}{\sqrt{x} \log x}$ for some absolute constant $c>0$.  Thus, the first moment $\mathbf{E} |A|$ of the size of the random set
$$ A \coloneqq \{ a \in [x]: p^2 \nmid n+a \text{ for all } p \leq \sqrt{x}\}$$
is at least $\frac{6}{\pi^2}x + \frac{c\sqrt{x}}{\log x}$.  Thus we can find an $n$ for which
$$ |A| \geq \frac{6}{\pi^2}x + \frac{c\sqrt{x}}{\log x}.$$
By construction, this set avoids a residue class modulo $p^2$ for all $p \leq \sqrt{x}$. On the other hand, for all $p > \sqrt{x}$ we see that $A \subseteq [x]$ avoids the residue class $\lfloor x \rfloor +1 \pmod{p^2}$. Hence $A$ is admissible, giving the lower bound
$$ A(x) \geq \frac{6}{\pi^2}x + \frac{c\sqrt{x}}{\log x}$$
as required.

\begin{remark}  It seems likely that by calculating higher moments of $|A|$ one could obtain some slight improvements to the lower bound (thereby potentially determining whether $A(x) > |\SF \cap [x]|$ holds for all $x \ge 18$), but we will not perform these calculations here.
\end{remark}

\appendix

\section{A large sieve inequality for square moduli}

We recall the large sieve inequality of Montgomery and Vaughan \cite{mv}:

\begin{theorem} \label{largesieve} Let $M$ and $N \ge 1$ be integers, and let $A$ be a subset of the interval $[M+1,M+N]$ that avoids $\omega(p) < p$ residue classes modulo $p$ for each prime $p$.  Then for any $Q \geq 1$, one has
$$ |A| \leq \frac{N+Q^2}{\sum_{q \leq Q} h(q)}$$
where
$$ h(q) \coloneqq \mu^2(q) \prod_{p|q} \frac{\omega(p)}{p-\omega(p)}.$$
\end{theorem}

Indeed, this is \cite[Corollary 1]{mv} after replacing \cite[(1.6)]{mv} in the proof of that corollary with \cite[(1.4)]{mv}. In this appendix we prove the following variant to Theorem \ref{largesieve}:

\begin{theorem}\label{sqsieve}  Let $M$ and $N \ge 1$ be integers, and let $A$ be a subset of the interval $[M+1,M+N]$ that avoids $\omega(p^2) < p^2$ residue classes modulo $p^2$ for each prime $p$.  Then for any $Q \geq 1$, one has
$$ |A| \leq \frac{N+Q^4}{\sum_{q \leq Q} h(q)}$$
where
$$ h(q) \coloneqq \mu^2(q) \prod_{p|q} \frac{\omega(p^2)}{p^2-\omega(p^2)}.$$
\end{theorem}

This result was essentially established in \cite{gallagher} (with a slightly different bound) under the additional assumption that the $\omega(p^2)$ residue classes removed had a tree structure, coming from first removing some residue classes modulo $p$, and then removing the same number of residue classes modulo $p^2$ for each surviving residue class modulo $p$.  However, we cannot assume this tree structure in our applications, so we need Theorem \ref{sqsieve} as a substitute.  One could also generalize this result to higher powers than squares, with the obvious modifications to the statement of the theorem; we leave this to the interested reader.

\begin{proof} Let $a_n$ be any sequence of complex numbers supported on $A$, and form the exponential sums $S(\alpha) = \sum_{n \in A} a_n e(n\alpha)$ with $e(\alpha) \coloneqq e^{2\pi i \alpha}$.  We then claim the inequality
\begin{equation}\label{qsum}
\sum_{\substack{q|q' \\ q'|q^2}} \sum_{\substack{1 \leq a \leq q' \\ {(a,q')=1}}} \left| S\left(\frac{a}{q'}+\alpha\right)\right|^2 \geq h(q) |S(\alpha)|^2
\end{equation}
for any squarefree $q$ and any $\alpha \in \R$.  By the usual Chinese remainder theorem argument (see, e.g., \cite[p. 493]{gallagher}), if this inequality is proven for $q=q_1$ and $q=q_2$ for some coprime $q_1,q_2$, then it also holds for $q=q_1q_2$.  As one can check \ref{qsum} by hand if $q = 1$, it therefore suffices to verify the inequality for $q=p$ a prime, in which case the outer sum only contains the two terms corresponding to $q' = p$ and $q' = p^2$.  By modulating the sequence $a_n$ by $e(n\alpha)$ we may furthermore normalize $\alpha=0$, so that we get
\begin{align*}
\sum_{\substack{q|q' \\ q'|q^2}} \sum_{\substack{1 \leq a \leq q' \\ {(a,q')=1}}} \left| S\left(\frac{a}{q'}+\alpha\right)\right|^2 &= \sum_{1 \leq a < p} \left| S\left(\frac{a}{p}\right)\right|^2 + \sum_{\substack{1 \leq a < p^2 \\ a \nmid p}} \left| S\left(\frac{a}{p^2}\right)\right|^2 \\
&= \sum_{1 \leq a < p^2} \left| S\left(\frac{a}{p^2}\right)\right|^2.
\end{align*}
Since $h(p) = \frac{\omega(p^2)}{p^2-\omega(p^2)}$, in order to deduce \ref{qsum} it remains to prove
\begin{equation} \label{qsumsimplified} 
\sum_{1 \leq a < p^2} \left| S\left(\frac{a}{p^2}\right)\right|^2 \geq \frac{\omega(p^2)}{p^2-\omega(p^2)} |S(0)|^2.
\end{equation}

Now, if we let $\nu(n) \coloneqq \omega(p^2) - p^2 1_S$, where $S$ is the union of the $\omega(p^2)$ residue classes modulo $p^2$ one is removing, then $\nu(n)$ is $p^2$-periodic and of mean zero, hence admits a Fourier expansion
$$ \nu(n) = \sum_{1 \leq a < p^2} c_a e(an/p^2)$$
for some coefficients $c_a$.  By the Plancherel identity and a brief calculation one has
$$ \sum_{1 \leq a < p^2} |c_a|^2 = (p^2 - \omega(p^2)) \omega(p^2),$$
and so by Cauchy--Schwarz we have
$$ \left|\sum_{1 \leq a < p^2} c_a S\left(\frac{a}{p^2}\right)\right|^2 \leq (p^2 - \omega(p^2)) \omega(p^2) \sum_{1 \leq a < p^2} \left| S\left(\frac{a}{p^2}\right)\right|^2.$$
However, the left-hand side can be rewritten as
$$ \left|\sum_{n \in A} \nu(n) a_n\right|^2 = \omega(p^2)^2 |S(0)|^2$$
so that \ref{qsumsimplified} and hence \ref{qsum} follow.

Summing \eqref{qsum} with $\alpha=0$ and squarefree $q \leq Q$, and $a_n = 1_A(n)$, we conclude that
$$ \sum_{q \leq Q} \mu^2(q) \sum_{\substack{q|q' \\ q'|q^2}} \sum_{\substack{1 \leq a \leq q' \\ (a,q')=1}} \left| S\left(\frac{a}{q'}\right)\right|^2 \geq |A| \sum_{q \leq Q} h(q).$$
The fractions $\frac{a}{q'}$ in the sum on the left-hand side are $\frac{1}{Q^4}$-separated, and so by the arithmetic large sieve inequality \cite[Theorem 1]{mv} one has
$$ \sum_{q \leq Q} \mu^2(q) \sum_{\substack{q|q' \\ q'|q^2}} \sum_{\substack{1 \leq a \leq q' \\ (a,q')=1}} \left| S\left(\frac{a}{q'}\right)\right|^2  \leq N + Q^4,$$
finishing the proof.
\end{proof}

\bibliographystyle{plainurl}
\bibliography{1102}

\end{document}